\documentclass[a4paper,12pt]{amsart}
\usepackage{amssymb,amstext,amsmath,amscd,amsthm,
amsfonts,enumerate,graphicx,latexsym}
\usepackage{url}
\usepackage[usenames]{color}
\definecolor{gr}{rgb}{0.1, .5 , .10}
\usepackage[all]{xy}

\newtheorem{theorem}{Theorem}[section]
\newtheorem{theorem*}{Theorem}
\newtheorem{corollary}[theorem]{Corollary}
\newtheorem{corollary*}[theorem*]{Corollary}
\newtheorem{lemma}[theorem]{Lemma}
\newtheorem{proposition}[theorem]{Proposition}

\theoremstyle{definition}
\newtheorem{definition}[theorem]{Definition}
\newtheorem{remark}[theorem]{Remark}

\newtheorem*{question*}{Question}

\newtheorem*{conjecture*}{Conjecture}
\newtheorem{example}[theorem]{Example}
\newtheorem{notation}[theorem]{Notation}
\newtheorem*{notation*}{Notation}

\newtheorem*{claim*}{Claim}

\numberwithin{equation}{theorem}
\def\Ext{\operatorname{Ext}}
\def\Tor{\operatorname{Tor}}
\def\rad{\operatorname{rad}}

\def\Hom{\operatorname{Hom}}

\def\End{\operatorname{End}}
\def\RHom{\operatorname{\mathsf{RHom}}}
\def\Ltensor{\otimes^\mathsf{L}}
\def\Ker{\operatorname{Ker}}

\def\thick{\operatorname{\mathsf{thick}}}
\def\add{\operatorname{\mathsf{add}}}

\def\D{\operatorname{\mathsf{D}}}

\def\smod{\operatorname{\underline{\mathsf{mod}}}}
\def\proj{\operatorname{\mathsf{proj}}}

\def\Kb{\mathsf{K^b}}
\def\ds{\mathsf{D_{sg}}}
\def\per{\operatorname{\mathsf{per}}}
\def\Tria{\operatorname{\mathsf{Tria}}}

\def\pd{\operatorname{pd}}

\def\silt{\operatorname{\mathsf{silt}}}

\newcommand{\dsilt}[2]{\operatorname{\mathsf{#1_{\mbox{{\scriptsize $#2$}}}-silt}}}
\def\A{\mathcal{A}}
\def\B{\mathcal{B}}
\def\C{\mathcal{C}}

\def\DD{\mathcal{D}}
\def\T{\mathcal{T}}
\def\U{\mathcal{U}}
\def\X{\mathcal{X}}
\def\Y{\mathcal{Y}}

\def\H{\mathcal{H}}

\def\ZZ{\mathcal{Z}}

\renewcommand{\S}{\mathcal{S}}
\def\RR{\mathcal{R}}
\def\KK{\mathcal{K}}

\def\Z{\mathbb{Z}}
\newcommand{\op}{{\rm op}}

\newcommand{\old}[1]{{\color{red} #1}}

\begin{document}
\setlength{\baselineskip}{15pt}
\title{When is the silting-discreteness inherited?}
\author{Takuma Aihara}
\address{Department of Mathematics, Tokyo Gakugei University, 4-1-1 Nukuikita-machi, Koganei, Tokyo 184-8501, Japan}
\email{aihara@u-gakugei.ac.jp}
\author{Takahiro Honma}
\address{Department of General Education, National Institute of Technology (KOSEN), Yuge College, Ochi, Ehime 794-2506, Japan}
\email{99cfqqc9@gmail.com}

\thanks{TA was partly supported by JSPS Grant-in-Aid for Young Scientists 19K14497.}
\begin{abstract}
We explore when the silting-discreteness is inherited.
As a result, one obtains that taking idempotent truncations and homological epimorphisms of algebras transmit the silting-discreteness.
We also study classification of silting-discrete simply-connected tensor algebras and silting-indiscrete selfinjective Nakayama algebras.
This paper contains two appendices; one states that every derived-discrete algebra is silting-discrete, and the other is about triangulated categories whose silting objects are tilting. 
\end{abstract}
\maketitle
\section{Introduction}

Silting-discreteness is a finiteness condition on a triangulated category, namely, a silting-discrete triangulated category admits only finitely many silting objects in any interval of silting objects \cite{Ai};
actually, the set of silting objects has a poset structure \cite{AI}.
In the case, we can fully grasp the whole picture of silting objects, and such a triangulated category has so nice structure.
For example,
a silting-discrete triangulated category satisfies a Bongartz-type condition; that is, any presilting object is partial silting \cite{AM}.
For an ST-pair $(\C, \DD)$,
$\C$ is silting-discrete if and only if $\DD$ is $t$-discrete if and only if every bounded $t$-structure in $\DD$ is algebraic; moreover, in the case, the stability space of $\DD$ is contractible \cite{AMY}; see also \cite{PSZ}.

In general, it is very hard to check if a given triangulated category is silting-discrete.
Thus, an extremely powerful tool was introduced \cite{AAC, Ai, AM}, and it is applied to the perfect derived category of a finite dimensional algebra over an algebraically closed field;
we know several algebras with silting-discrete perfect derived categories, say \emph{silting-discrete} algebras \cite{AAC, AdK, Ai, AI, AHMW, AM, Au, AD, BPP, EJR}.
Typical examples of silting-discrete algebras are local algebras, piecewise hereditary algebras of Dynkin type and representation-finite symmetric algebras.
It was also showed by \cite{YY} that for an ST-triple $(\C, \DD, T)$ if $\C$ is $T$-discrete, then it is silting-discrete, which imples that any derived-discrete algebra is silting-discrete;
we will also observe the result in the appendix.
We obtain from \cite{AMY} that if $(\C, \DD)$ is a $(d+1)$-Calabi--Yau pair for $d=1,2$, then $\C$ is silting-discrete if and only if the endomorphism algebra of a silting object is $\tau$-tilting finite, which leads to the fact that for the derived preprojective algebra $\Gamma:=\Gamma_{d+1}(Q)$ of a finite quiver $Q$ with $H^0(\Gamma)$ finite dimensional (note that it is a nonpositive dg algebra), the perfect derived category of $\Gamma$ is silting-discrete if and only if $Q$ is Dynkin.

The aim of this paper is to investigate when the silting-discreteness is inherited.
For example, derived equivalences certainly transmit the silting-discreteness.
Tensoring a local algebra also conveys the property \cite{AH}.
One of the main theorems of this paper (Theorem \ref{fs}) states that a full triangulated subcategory of a silting-discrete triangulated category is also silting-discrete.
Moreover, we show that a homological epimorphism of (nonpositive) dg algebras induces the silting-discreteness (Theorem \ref{recollement}).
These lead us to the following results.


\begin{theorem*}[Corollaries \ref{idempotent} and \ref{stratifying}]\label{main1}
Let $\Lambda$ be a finite dimensional algebra over an algebraically closed field. 
Assume that $\Lambda$ is silting-discrete.
Then we have the following.
\begin{enumerate}
\item For an idempotent $e$ of $\Lambda$, the truncation $e\Lambda e$ is silting-discrete.
\item 
If $e$ is a stratifying idempotent of $\Lambda$, then $\Lambda/\Lambda e\Lambda$ is silting-discrete.
\end{enumerate}
\end{theorem*}



We remark that quotients of algebras always inherit the $\tau$-tilting finiteness \cite{DIRRT}, but not necessarily the silting-discreteness (Subsection \ref{sec:quot}).

The Bongartz-type condition is also one of the most important subjects in silting theory.
We naturally ask whether or not the Bongartz-type condition always holds\footnote{
Recently, a negative answer to this question was given \cite{LZ}. (See Subsection \ref{Bongartzfails2}.)
}.
For example, silting-discrete triangulated categories satisfy the Bongartz-type condition and every 2-term presilting object can be completed to a silting object \cite{AM, Ai}.
Note that the tilting version (i.e., is a pretilting object partial tilting?) often fails \cite{R, LVY}; see Example \ref{Bongartzfails}.
Theorem \ref{main1} and its proof provide a new Bongartz-type lemma under a certain condition;
see Corollary \ref{Bongartz} for a general setting, which recovers the 2-term version.

\begin{theorem*}[Corollary \ref{heDG}(2)]
Let $e$ be an idempotent of $\Lambda$.
Then a presilting object of the perfect derived category of $\Lambda$ generating the thick closure of $e\Lambda$ is partial silting.
\end{theorem*}

As a corollary of Theorem \ref{main1}, we classify silting-discrete simply-connected tensor algebras; note that the tensor algebra of three nonlocal algebras is never silting-discrete \cite{AH}.
Also, we obtain a surprising result, which says that a representation-finite selfinjective algebra is not necessarily silting-discrete.

\begin{theorem*}[Proposition \ref{simplyconnected}, Theorems \ref{sdscta} and \ref{Nakayama}]
\begin{enumerate}
\item A simply-connected algebra is silting-discrere if and only if it is piecewise hereditary of Dynkin type.
\item Let $A$ and $B$ be finite dimensional (nonlocal) simply-connected algebras over an algebraically closed field $K$.
Then the following are equivalent:
\begin{enumerate}
\item $A\otimes_KB$ is silting-discrete;
\item it is piecewise hereditary of type $D_4, E_6$ or $E_8$;
\item it is derived equivalent to a commutative ladder of degree $\leq4$. 
\end{enumerate}
\item The selfinjective Nakayama algebra with $n$ simple modules of length $r$ is not silting-discrete if (i) $r=3, 4$ and $n\geq11$; (ii) $r=5,6$ and $n\geq r+8$; or (iii) $r\geq7$ and $n\geq 2r+1$.

\end{enumerate}
\end{theorem*}



\section{Main results}

Throughout this paper, let $\T$ be a triangulated category which is Krull--Schmidt, $K$-linear for an algebraically closed field $K$ and Hom-finite.
For an object $T$ of $\T$, its thick closure is denoted by $\thick T$; i.e., it is the smallest thick subcategory of $\T$ containing $T$.

Silting objects play a central role in this paper.

\begin{definition}
Let $T$ be an object of $\T$.
\begin{enumerate}
\item We say that $T$ is \emph{presilting} (\emph{pretilting}) if it satisfies $\Hom_\T(T, T[i])=0$ for any $i>0\ (i\neq0)$.
\item A presilting object $T$ is said to be \emph{silting} (\emph{tilting}) provided $\T=\thick T$.
\item A \emph{partial silting} (\emph{partical tilting}) object is defined to be a direct summand of some silting (tilting) object.
\end{enumerate}
\end{definition}


\begin{notation}
We use the following notations.
\begin{itemize}
\item Denote by $\silt\T$ the set of isomorphism classes of basic silting objects of $\T$.
\item We always suppose that $\silt\T$ is nonempty.
\item For a silting object $A$ of $\T$ and $d>0$, $\dsilt{d}{A}\T:=\{T\in\silt\T\ |\ A\geq T\geq A[d-1] \}$.
\item Let $S$ be a presilting object of $\T$.
The subset $\silt_S\T$ of $\silt\T$ consists of all silting objects having $S$ as a direct summand.
\end{itemize}
\end{notation}

As terminology, for presilting objects $T$ and $U$ of $\T$ we use $T\geq U$ if $\Hom_\T(T, U[i])=0$ for every $i>0$.
This actually gives a partial order on $\silt\T$ \cite[Theorem 2.11]{AI}.

We define the silting-discreteness of $\T$, which is the main theme of this paper.

\begin{definition}
\begin{enumerate}
\item We say that $\T$ is \emph{silting-discrete} if it has a silting object $A$ such that for any $d>0$, the set $\dsilt{d}{A}\T$ is finite; see \cite[Proposition 3.8]{Ai} for equivalent conditions of silting-discreteness.
(As a convention, a triangulated category without silting object is also called silting-discrete.)
\item For a silting object $A$ of $\T$, we call $\T$ \emph{$2_A$-silting finite} if $\dsilt{2}{A}\T$ is a finite set.
\item A \emph{2-silting finite} triangulated category is defined to be $2_A$-silting finite for every silting object $A$ of $\T$.
\end{enumerate}
\end{definition}

\begin{remark}
\begin{enumerate}
\item The notion of $2_A$-silting finiteness is nothing but that of $\tau$-tilting finiteness for $\End_\T(A)$ \cite[Corollary 2.8]{DIJ}.
For $\tau$-tilting theory, we refer to \cite{AIR}.
\item The silting-discreteness is preserved under triangle equilvaneces, but the $2_A$-silting finiteness is not always so;
it depends on the choice of silting objects $A$.
\end{enumerate}
\end{remark}

Let us recall several results in \cite{Ai, AM}.

\begin{proposition}\label{AM}
Let $A$ be a silting object and $T$ a presilting object of $\T$.
\begin{enumerate}
\item \cite[Proposition 2.16]{Ai} If $A\geq T\geq A[1]$, then $T$ is a direct summand of a silting object lying in the interval between $A$ and $A[1]$.
\item \cite[Proposition 2.14]{AM} Assume that $\T$ is $2_A$-silting finite.
If $A\geq T\geq A[2]$, then there exists a silting object $B$ of $\T$ satisfying $A\geq B\geq T\geq B[1]\geq A[2]$.
\item \cite[Proposition 2.14]{AM} Assume that $\T$ is silting-discrete.
Then there is a silting object $P$ of $\T$ with $P\geq T\geq P[1]$.
\item \cite[Theorem 2.4]{AM} $\T$ is silting-discrete if and only if it is 2-silting finite.
\end{enumerate}
\end{proposition}

In the following, we investigate the silting-discreteness of triangulated categories:
\begin{itemize}
\item Full triangulated subcategories inherit the silting-discreteness (\ref{subsec:fts});
\item Recollements make an influence on silting-discreteness (\ref{subsec:recollement}).
\end{itemize}

Moreover, we will also consider silting-discrete algebras as these corollaries.

Let $\Lambda$ be a finite dimensional $K$-algebra which is basic and ring-indecomposable, and modules are always finite dimensional and right unless otherwise noted.
The perfect derived category of $\Lambda$ is denoted by $\Kb(\proj\Lambda)$.

We say that $\Lambda$ is \emph{silting-discrete} if so is $\Kb(\proj\Lambda)$.
Here is an example of silting-discrete algebras which we should first observe; one frequently uses it in the rest of this paper.

\begin{example}\label{sdpha}
A piecewise hereditary algebra of type $\H$ is silting-discrete if and only if $\H$ is of Dynkin type.
\end{example}

Let $Q$ be a (finite) quiver.
For an arrow $\alpha: i\to j$ of $Q$, we consider the \emph{formal inverse} $\alpha^{-1}:j\to i$ and define the \emph{vertual degree} by $\deg(\alpha^\pm)=\pm1$ (double-sign corresponds).
The vertual degree of a walk (a sequence of successive arrows and formal inverses) is defined by $\log$-rule: $\deg(\alpha\beta)=\deg(\alpha)+\deg(\beta)$.
We say that $Q$ is \emph{gradable} if any closed walk has vertual degree 0.
For instance, every tree quiver (i.e., the underlying graph is tree) is gradable.
Note that a gradable quiver admits no oriented cycle; but the converse does not necessarily hold.

We observe silting-discrete algebras with radical square zero.

\begin{proposition}\label{gradablersz}
Let $\Lambda$ be the radical-square-zero algebra given by a gradable quiver $Q$.
Then $\Lambda$ is silting-discrete if and only if $Q$ is Dynkin. 
\end{proposition}
\begin{proof}
It follows from \cite[Theorem 4.4]{Y} that $\Lambda$ is derived equivalent to $KQ$.
\end{proof}

Note that there is a silting-discrete radical-square-zero algebra presented by an ungradable nonDynkin acyclic quiver (Example \ref{kronecker}).

We say that $\Lambda$ is \emph{gentle} if it is presented by a quiver $Q$ with relation $I$ such that
\begin{enumerate}
\item for any vertex of $Q$, there exist at most 2 incoming and at most 2 outgoing arrows;
\item for each arrow $x: i\to j$ of $Q$, the number of arrows $y$ stopping at $i$ (starting from $j$) with $yx\not\in I$ ($xy\not\in I$) is at most 1;
\item for each arrow $x: i\to j$ of $Q$, the number of arrows $y$ stopping at $i$ (starting from $j$) with $yx\in I$ ($xy\in I$) is at most 1;
\item all relations in $I$ are paths of length 2.
\end{enumerate}
A quiver $Q$ is called \emph{one-cycle} provided its underlying graph contains exactly one cycle.
A \emph{gentle one-cycle} algebra is defined to be a gentle algebra given by a one-cycle quiver.
We say that a gentle one-cycle algebra \emph{satisfies the clock condition} if the numbers of clockwise and counter-clockwise oriented relations on the cycle coincide.

The following proposition is also useful.

\begin{proposition}\label{goc}
A gentle one-cycle algebra $\Lambda$ is silting-discrete if and only if it does not satisfy the clock condition. 
\end{proposition}
\begin{proof}
Assume that $\Lambda$ satisfies the clock condition.
By \cite[Theorem (A)]{AS}, we see that $\Lambda$ is piecewise hereditary of type $\widetilde{A}$, which is not silting-discrete.

Assume that $\Lambda$ does not satisfy the clock condition.
It follows from \cite[Theorem]{V} that $\Lambda$ is derived-discrete, and so it is silting-discrete by \cite[Example 3.9, Corollary 4.2]{YY}.
\end{proof}


Let us observe the silting-discreteness of type $\widetilde{A}$ with commutative relation.

\begin{proposition}\label{AtildeC}
Let $\Lambda$ be the algebra given by the quiver of type $\widetilde{A_{p,q}}$ with commutative relation; suppose that $1<p\leq q$.
Then $\Lambda$ is silting-discrete if and only if $p=2$ and any $q$ or $p=3$ and $q=3, 4, 5$.
\end{proposition}
\begin{proof}
Let $Q$ be the quiver $\xymatrix{\bullet \ar[r]^{\alpha_2} & \cdots \ar[r] & \bullet \ar[r]^{\alpha_p} & \bigstar & \bullet \ar[l]_{\beta_q} & \cdots \ar[l] & \bullet \ar[l]_{\beta_2}}$.
Since $\Lambda$ is the one-point extension of the path algebra $KQ$ by the indecomposable injective module corresponding to the vertex $\bigstar$,
we see that it is derived equivalent to the one-point extension of $KQ$ by the indecomposable projective module corresponding to the vertex $\bigstar$ \cite[Theorem 1]{BL},
which is the path algebra of the quiver
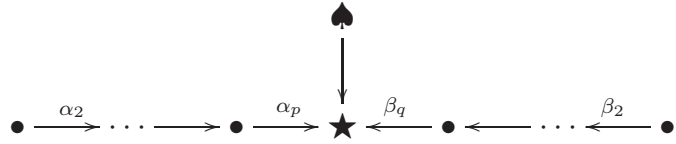
\[\xymatrix{
&&&\spadesuit \ar[d]&&& \\
\bullet \ar[r]^{\alpha_2} & \cdots \ar[r] & \bullet \ar[r]^{\alpha_p} & \bigstar & \bullet \ar[l]_{\beta_q} & \cdots \ar[l] & \bullet \ar[l]_{\beta_2}
}\]
Hence, it turns out that it is silting-discrete iff one of the desired conditions holds.
\end{proof}

\subsection{Full triangulated subcategories}\label{subsec:fts}

The following is the first main theorem.

\begin{theorem}\label{fs}
Let $\U$ be a full triangulated subcategory of $\T$.
If $\T$ is silting-discrete, then so is $\U$.
\end{theorem}
\begin{proof}
%
We show that $\U$ is 2-silting finite; hence, it is silting-discrete by Proposition \ref{AM}(4).
Let $T$ be a silting object in $\U$.
As $T$ is presilting in $\T$, we obtain a silting object $A$ satisfying $A\geq T\geq A[1]$ by Proposition \ref{AM}(3).
Let $U$ be a silting object of $\U$ with $T\geq U\geq T[1]$.
Since there is a triangle $T_1\to T_0\to U\to T_1[1]$ with $T_0, T_1\in\add T$ in $\U$ \cite[Proposition 2.24]{AI}, we have inequalities $A\geq U\geq A[2]$.
Proposition \ref{AM}(1) and (2) lead to the fact that $U$ is a direct summand of a silting object lying in the interval beween $A$ and $A[2]$.
As the interval is a finite set, we deduce that $\{U\in\silt\U\ |\ T\geq U\geq T[1] \}$ is also finite.
\end{proof}

A typical example of fully faithful functors is the triangle functor $-\otimes_{e\Lambda e}e\Lambda:\Kb(\proj e\Lambda e)\to\Kb(\proj\Lambda)$ for an idempotent $e$ of $\Lambda$.
This yields a corollary of Theorem \ref{fs}, which says that taking idempotent truncations brings the silting-discreteness from the original.

\begin{corollary}\label{idempotent}
Let $e$ be an idempotent of $\Lambda$.
If $\Lambda$ is silting-discrete, then so is $e\Lambda e$.
\end{corollary}

%
%

\subsection{Recollements}\label{subsec:recollement}

In this subsection, we investigate an impact of recollements on silting-discreteness.

Let us begin with the study of silting reductions \cite{IY},
which enable us to realize the Verdier quotient of $\T$ by the thick closure of a presilting object as a certain subfactor category of $\T$ and which make isomorphisms of silting posets.

Let $S$ be a presilting object of $\T$ and put $\S:=\thick S$.
The Verdier quotient of $\T$ by $\S$ is denoted by $\T/\S$.

Take a full subcategory $\ZZ$ of $\T$ as follows:
\[\ZZ:=\{X\in\T\ |\ \Hom_\T(X,S[i])=0=\Hom_\T(S, X[i])\ \mbox{for any } i>0 \}.\]
We denote by $\frac{\ZZ}{[S]}$ the additive quotient of $\ZZ$ modulo the ideal $[S]$, which admits a triangle structure \cite[Theorem 3.6]{IY}.

Then, the silting reduction \cite[Theorem 3.6]{IY} shows that there is a triangle equivalence between $\T/\S$ and $\frac{\ZZ}{[S]}$. Since $\T$ is Hom-finite and Krull--Schmidt, so is $\T/\S$.

Now, we prepare a lemma for the main theorem of this subsection. 

\begin{lemma}\label{verdier}
Keep the notations above.
If $\T$ is silting-discrete, then so is $\T/\S$.
\end{lemma}
\begin{proof}
By the silting reduction \cite[Theorem 3,7 and Corollary 3.8]{IY}, the natural functor $\T\to\T/\S$ induces an isomorphism $\silt_S\T\simeq\silt\T/\S$ of posets.
Thus, it is not hard to deduce that $\T/\S$ is silting-discrete.
\end{proof}

\begin{remark}
We are assuming that $\silt\T\neq\emptyset$ to apply the silting reduction, while $\silt\T/\S=\emptyset$ may occur; see Subsection \ref{Bongartzfails2}.
\end{remark}

To realize the goal of this subsection, we need (nonpositive) dg algebras. 

Let $\A$ be a dg $K$-algebra and $\per(\A)$ denote the perfect derived category of $\A$.
If $\A$ is nonpositive with finite dimensional cohomologies, then $\per(\A)$ is Hom-finite and Krull--Schmidt; see \cite[Proposition 6.12]{AMY} for example.

%
%
%

We give a remarkable example \cite[Example 5.3]{Y}, which mentions that the 0th cohomology $H^0(\A)$ of $\A$ does not necessarily inherit the silting-discreteness from $\A$; moreover, the endomorphism algebra of a silting object does not always inherit the property either.

\begin{example}\label{exampleY}
Let $\Lambda$ be the algebra presented by the quiver $Q$
\[\xymatrix{
  & 5 & 6 \ar[d] & \\
1 \ar[r]_\alpha & 2 \ar[r]_\beta\ar[u] & 3 \ar[r]_\gamma & 4 
}\]
with relation $\alpha\beta\gamma=0$.
Then, we easily observe that $\Lambda$ is piecewise hereditary of type $E_6$; so, it is silting-discrete.
Take $T:=S_1\oplus S_2[1]\oplus S_3[2]\oplus S_4[3]\oplus S_5[2]\oplus S_6[1]$, where $S_i$ is the simple module corresponding to the vertex $i$ of $Q$;
it is in $\Kb(\proj\Lambda)$ because $\Lambda$ has finite global dimension.
We see that $T$ is a silting object whose dg endomorphism algebra $\A$ is nonpositive with finite dimensional cohomologies;
so, $\per(\A)$ is Krull--Schmidt.
By Keller--Rickard's theorem \cite{K, R}, we have a triangle equivalence $\per(\A)\simeq\Kb(\proj\Lambda)$, whence $\per(\A)$ is silting-discrete.
However, the 0th cohomology $H^0(\A)$ of $\A$, which is isomorphic to $\End_{\Kb(\proj\Lambda)}(T)$, is the radical-square-zero algebra of $Q^\op$, and it is derived equivalent to a hereditary algebra of type $\widetilde{D_5}$.
Thus, $H^0(\A)$ is not silting-discrete.
\end{example}

This example also says that although a derived preprojective algebra $\A$ of Dynkin type (it is a nonpositive dg algebra) admits a silting-discrete perfect derived category \cite[Corollary 8.6]{AMY}, it is still open whether a preprojective algebra of Dynkin type except $A_2, D_{2n}, E_7$ and $E_8$, which appears as the 0th cohomology of $\A$, is silting-discrete or not.

We aim at our goal of this subsection.

Recall that a \emph{recollement} $\RR$ consists of the following data of triangulated categories $\DD, \X$ and $\Y$, and triangle functors between them:
\[
\xymatrix@C=3cm{
\Y \ar[r]|{i_*=i_!} & \DD \ar[r]|{j^!=j^*}\ar@/_2pc/[l]_{i^*}\ar@/^2pc/[l]^{i^!} & \X \ar@/_2pc/[l]_{j_!}\ar@/^2pc/[l]^{j_*},
}
\]
such that 
\begin{enumerate}
\item $(i^*, i_*),\ (i_!, i^!),\ (j_!, j^!)$ and $(j^*, j_*)$ are adjoint pairs;
\item $i_*\ (=i_!),\ j_!$ and $j_*$ are fully faithful;
\item any object $D$ of $\DD$ admits two triangles
\[\begin{array}{c@{\ \mbox{and }}c}
\mbox{$i_!i^!(D)\to D\to j_*j^*(D)\to i_!i^!(D)[1]$} &
\mbox{$j_!j^!(D)\to D\to i_*i^*(D)\to j_!j^!(D)[1]$},
\end{array}\]
where the morphisms around $D$ are given by adjunctions.
\end{enumerate}
Omitting the triangle functors, we write such a recollent as $\RR=(\Y, \DD, \X)$.

We denote by $\D(-)$ the (unbounded) derived catgory, and for an object $X$ of $\D(\Lambda)$, $\Tria(X)$ stands for the smallest full triangulated subcategory of $\D(\Lambda)$ containing $X$ which is closed under all direct sums.

From now, we consider either of the following situations.
\begin{itemize}
\item Start from a given recollement $(\Y, \D(\Lambda),\X)$ with $\X$ compactly generated: Then, by \cite[Theorem 4.5]{AHMV} $\X$ can be written as $\Tria(S)$ for some presilting object $S$ of $\Kb(\proj\Lambda)$.
\item Start from a given presilting object $S$ of $\Kb(\proj\Lambda)$: Since $\Tria(S)$ is a smashing subcategory of $\D(\Lambda)$ \cite[Proposition 3.5]{AHMV}, there exists a full triangulated subcategory $\Y$ of $\D(\Lambda)$ such that $(\Y, \D(\Lambda), \Tria(S))$ forms a recollement;
actually, $\Y$ is right perpendicular to $\Tria(S)$ \cite[Corollary 2.4]{NS}.
\end{itemize}
In any case, we have a recollement $(\Y,\D(\Lambda),\Tria(S))$ for a presilting object $S$ of $\Kb(\proj\Lambda)$.
Furthermore, thanks to Nicolas--Saorin's result \cite[Section 4]{NS}, there is a dg $K$-algebra $\A$ and a morphism $\varphi:\Lambda\to\A$ of dg algebras such that $i_!=-\Ltensor_\A\A$ induces a recollement $(\D(\A), \D(\Lambda),\Tria(S))$.
Here, $\varphi$ is embedded into the triangle $X:=\RHom_\Lambda(S, \Lambda)\Ltensor_\B S\to \Lambda\xrightarrow{\varphi}\A\to X[1]$ in $\D(\Lambda^\op\otimes_K\Lambda)$, where $\B$ is the dg endomorphism algebra of $S$.
We evidently observe that $H^i(\A)$ is finite dimensional for any $i$.
One can also describe $\A$ as the dg endomorphism algebra of $\Lambda$ in a dg enhancement of $\Kb(\proj\Lambda)/\thick(S)$.
%
Such a morphism $\varphi$ is called a \emph{homological epimorphism}; see \cite{P}.

Now, we state the second main theorem of this paper.


\begin{theorem}\label{recollement}
Keep the notations above.
If $\Lambda$ is silting-discrete, then so is $\per(\A)$.
\end{theorem}

A key point is an anologue of the proof of \cite[Corollary 2.12(a)]{KY}.

\begin{proof}
The recollement $(\D(\A), \D(\Lambda), \Tria(S))$ as in before Theorem \ref{recollement} gives rise to an exact sequence $0\to\Tria(S)\to\D(\Lambda)\xrightarrow{i^*}\D(\A)\to0$ of triangulated categories.
As $S$ is presilting in $\Kb(\proj\Lambda)$, we derive from Thomason--Trobaugh--Yao localization theorem \cite[Theorem 2.1]{N} that $i^*$ induces a triangule equivalence $\Kb(\proj\Lambda)/\thick(S)\simeq\per(\A)$,
whence $\per(\A)$ is silting-discrete  by Lemma \ref{verdier}.
\end{proof}

\begin{remark}
Theorem \ref{recollement} holds enoughly for a nonpositive dg $K$-algebra $\Lambda$.
\end{remark}

The proof of Theorem \ref{recollement} spins off the following result, which mentions a Bongartz-type lemma for a given presilting object $S$;
if $\Lambda\geq S\geq \Lambda[1]$, then $\A$ is nonpositive because $H^i(\A)\simeq H^{i+1}(X)$, so the result gives a generalization of Proposition \ref{AM}(1).

\begin{corollary}\label{Bongartz}
Keep the notations as in before Theorem \ref{recollement}.
Assume that $\A$ is nonpositive.
Then $S$ is partial silting.
\end{corollary}
\begin{proof}
In the proof of Theorem \ref{recollement}, we got a triangle equivalence $\Kb(\proj\Lambda)/\thick S\simeq \per(\A)$.
The silting reduction (Lemma \ref{verdier}) yields isomorphisms $\silt_S(\Kb(\proj\Lambda))\simeq\silt(\Kb(\proj\Lambda)/\thick S)\simeq \silt(\per(\A))$ of posets.
Since $\A$ is nonpositive, it is silting in $\per(\A)$, whence $\silt(\per(\A))$ is nonempty; so is $\silt_S(\Kb(\proj\Lambda))$.
Thus, it turns out that $S$ is partial silting.
%
\end{proof}

Let $e$ be an idempotent of $\Lambda$.
Following \cite[Proposition 2.10]{KY} and its proof,
we can construct a nonpositive dg $K$-algebra $\A_e$ and a homological epimorphism $\varphi:\Lambda\to\A_e$ of dg algebras;
embed the canonical morphism $\Lambda e\Ltensor_{e\Lambda e}e\Lambda\to\Lambda$ into the triangle
\[\Lambda e\Ltensor_{e\Lambda e}e\Lambda\to\Lambda\xrightarrow{\varphi'}\A'\to\Lambda e\otimes_{e\Lambda e}^\mathsf{L}e\Lambda[1],\]
and take the standard truncations $\A_e:=\sigma^{\leq0}(\A')$ and $\varphi:=\sigma^{\leq0}(\varphi')$.
Then, $\A_e$ and $\varphi$ are the desired dg algebra and homological epimorphism, since we have isomorphims
\begin{equation}\label{cohomologies}
H^i(\A_e)\simeq\begin{cases}
\ 0 & (i>0); \\
\ \Lambda/\Lambda e\Lambda & (i=0); \\
\ \Ker(\Lambda e\otimes_{e\Lambda e}e\Lambda\to\Lambda e\Lambda) & (i=-1); \\
\ \Tor^{e\Lambda e}_{-i-1}(\Lambda e, e\Lambda) & (i<-1).
\end{cases}
\end{equation}
Here is a corollary of Theorem \ref{recollement} and Corollary \ref{Bongartz}.

\begin{corollary}\label{heDG}
Let $e$ be an idempotent of $\Lambda$.
Then we have the following.
\begin{enumerate}
\item If $\Lambda$ is silting-discrete, then so is $\per(\A_e)$.
\item A presilting object of $\Kb(\proj\Lambda)$ generating $\thick(e\Lambda)$ is partial silting.
\end{enumerate}
\end{corollary}
\begin{proof}
As above, we obtain a recollement $(\D(\A_e), \D(\Lambda), \Tria(e\Lambda))$ and know that $\A_e$ is nonpositive.
So, the first assertion directly follows from Theorem \ref{recollement}.
If $S$ is a presilting object of $\Kb(\proj\Lambda)$ generating $\thick(e\Lambda)$, then we have $\Tria(S)=\Tria(e\Lambda)$, whence the last assertion is derived from Corollary \ref{Bongartz}.
\end{proof}

%

In the rest of this subsection, we give an explicit description of $\A_e$ and its example thanks to \cite{KY2}.

We first construct a nonpositive dg algebra $\widetilde{\Lambda}$ quasi-isomorphic to $\Lambda$; see \cite[Construction 2.6]{O}.
Let $\Lambda$ be presented by a (finite) quiver $Q$ with admissible ideal $I$; write $Q^{(0)}:=Q$.
We make a graded quiver $Q^{(1)}$ by adding to $Q^{(0)}$ arrows in degree $-1$ which correspond to each element of $I$ and consider a differential which sends new arrows to their corresponding relations; then $H^0(KQ^{(1)})=\Lambda$.
Picking out a generating set of $H^{-1}(KQ^{(1)})$,
one produces a graded quiver $Q^{(2)}$ by adding to $Q^{(1)}$ arrows in degree $-2$ which correspond to each element of the generating set and consider a differential as well.
Iterating this, we get a graded quiver $Q^{(\infty)}:=\bigcup_{i=0}^\infty Q^{(i)}$ such that the dg quiver algebra $\widetilde{\Lambda}:=(KQ^{(\infty)}, d)$ is quasi-isomorphic to $\Lambda$.

Now, we apply Theorem 7.1 of \cite{KY2} to obtain that $\A_e$ is quasi-isomorphic to $\widetilde{\Lambda}/\widetilde{\Lambda}\widetilde{e}\widetilde{\Lambda}$, where $\widetilde{e}$ is the idempotent of $\widetilde{\Lambda}$ corresponding to $e$.

\begin{example}\label{kronecker}
Let $\Lambda$ be the algebra presented by the quiver 
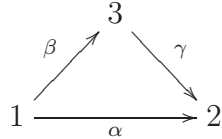
\[\xymatrix{
 & 3 \ar[dr]_(0.4){}="b"^\gamma &  \\
1 \ar[ru]_(0.6){}="a"^\beta \ar[rr]_\alpha & & 2
}\]
with $\beta\gamma=0$; it is silting-discrete by Proposition \ref{goc}.
Then, $\widetilde{\Lambda}$ is the dg quiver algebra of
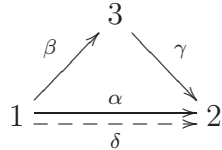
\[\xymatrix{
 & 3 \ar[dr]_(0.4){}="b"^\gamma &  \\
1 \ar[ru]_(0.6){}="a"^\beta \ar@<2pt>[rr]^\alpha \ar@<-2pt>@{-->}[rr]_\delta & & 2
}\]
with $\deg(\delta)=-1$, the other arrows of degree 0 and the trivial differential.

Let $e$ be the primitive idempotent of $\Lambda$ corresponding to the vertex 3.
We see that $\A_e$ is the dg quiver algebra of the graded Kronecker quiver with arrows of degree 0 and $-1$, and the trivial differential.
This is silting-discrete by Corollary \ref{heDG}(1). 
\end{example}

Note that the dg quiver algebra of the graded Kronecker quiver with arrows of the same degree $n\leq0$ and the trivial differential is not silting-discrete; indeed, it is derived equivalent to the (ordinary) Kronecker algebra $\xymatrix{1 \ar@<2pt>[r]\ar@<-2pt>[r] & 2}$ by the silting object $P_1[-n]\oplus P_2$.
Here, $P_1$ and $P_2$ are the indecomposable projective modules corresponding to the vertices 1 and 2, respectively.

This observation gives an example of silting-indiscrete algebras.


\begin{example}\label{Atilde2}
Let $Q$ be the quiver of type $\widetilde{A_{p,q}}$; assign to the arrows of two paths (length $p$ and $q$, respectively) from the source to the sink $\alpha$ and $\beta$, respectively.
Then $\Lambda:=KQ/(\alpha^p, \beta^q)$ is silting-indiscrete.
\end{example}
\begin{proof}
We see that $\widetilde{\Lambda}$ is given by the graded quiver
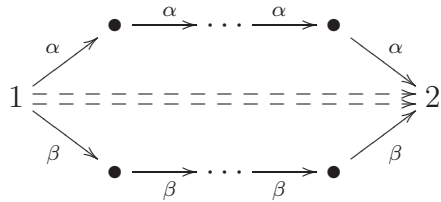
\[\xymatrix@R=5mm{
 & \bullet \ar[r]^\alpha & \cdots \ar[r]^\alpha & \bullet \ar[dr]^\alpha & \\
1 \ar[ru]^\alpha\ar[dr]_\beta \ar@<2pt>@{-->}[rrrr]\ar@<-2pt>@{-->}[rrrr] & &&& 2 \\
 & \bullet \ar[r]_\beta & \cdots \ar[r]_\beta & \bullet \ar[ur]_\beta &
}\]
with the dashed arrows of degree $-1$.
Taking the sum $e$ of the primitive idempotents of $\Lambda$ corresponding to all but the vertices 1 and 2, it is obtained that $\widetilde{\Lambda}/\widetilde{\Lambda}\widetilde{e}\widetilde{\Lambda}$ is the dg quiver algebra of the graded Kronecker quiver with arrows of degree $-1$,  which is silting-indiscrete, so is $\Lambda$ by Corollary \ref{heDG}(1).
\end{proof}

\subsection{Homological epimorphisms of algebras}



In this subsection, we restrict the results in Subsection \ref{subsec:recollement} to `ordinary' algebras; i.e., $\A$ is just an algebra.

Let $\Gamma$ be a finite dimensional $K$-algebra.
Recall that a homomorphism $\varphi:\Lambda\to\Gamma$ of algebras is a \emph{homological epimorphism} if the canonical morphism $\Gamma\Ltensor_\Lambda\Gamma\to\Gamma$ is an isomorphism in the derived category of $(\Gamma, \Gamma)$-bimodules; that is, the multiplication map $\Gamma\otimes_\Lambda\Gamma\to\Gamma$ is isomorphic and $\Tor^\Lambda_i(\Gamma,\Gamma)=0$ for all $i>0$.
Note that the former condition $\Gamma\otimes_\Lambda\Gamma\simeq\Gamma$ is satisfied iff $\varphi$ is an epimorphism in the category of rings; it is called a \emph{ring-epimorphism} \cite[Proposition 1.1]{S}.
A surjective homomorphism (i.e., an `ordinary' epimorphism) is a ring-epimorphism, but the converse is not always true.

As a homological epimorphism $\Lambda\to \Gamma$ of algebras induces a recollement $(\D(\Gamma), \D(\Lambda), \X)$ by \cite[Section 4]{NS}, we immediately obtain the following result from Theorem \ref{recollement} under a subtle condition.


\begin{corollary}\label{he}
Let $\varphi:\Lambda\to\Gamma$ be a homological epimorphism and assume that $\X$ is compactly generated; e.g., $\pd\Gamma_\Lambda<\infty$. 
If $\Lambda$ is silting-discrete, then so is $\Gamma$.
\end{corollary}

We say that an idempotent $e$ of $\Lambda$ (or the ideal $\Lambda e\Lambda$) is \emph{stratifying} if the canonical homomorphism $\pi:\Lambda\to\Lambda/\Lambda e\Lambda$ is a homological epimorphism;
equivalently, the dg algebra $\A_e$ as in Corollary \ref{heDG} is just isomorphic to $\Lambda/\Lambda e\Lambda$.
For example, if $\Lambda$ is hereditary (i.e., $\Lambda=KQ$ for an acyclic quiver $Q$), then any idempotent $e$ of $\Lambda$ is stratifying; notice that $\widetilde{\Lambda}$ as in the last of Subsection \ref{subsec:recollement} is nothing but $(KQ, 0)=\Lambda$.
More generally, if $\Lambda e\Lambda$ is projective as a right $\Lambda$-module, then it is stratifying.
Here is a direct consequence of Corollaries \ref{heDG}/\ref{he}. 

\begin{corollary}\label{stratifying}
Let $e$ be a stratifying idempotent of $\Lambda$.
If $\Lambda$ is silting-discrete, then so is $\Lambda/\Lambda e\Lambda$.
\end{corollary}

A typical example of stratifying idempotents is a primitive idempotent corresponding to a source or a sink in the Gabriel quiver of $\Lambda$. So, we have:

\begin{example}\label{sourcesink}
Let $e$ be a primitive idempotent corresponding to a source or a sink in the Gabriel quiver of $\Lambda$.
If $\Lambda$ is silting-discrete, then so is $\Lambda/\Lambda e\Lambda$.
In other words, if a one-point extension algebra of $\Lambda$ is silting-discrete, then so is $\Lambda$.
\end{example}

We remark that this example can be also got from the fact that $\Lambda/\Lambda e\Lambda$ is isomorphic to $(1-e)\Lambda(1-e)$ for such an idempotent $e$.

Note that the converse of Example \ref{sourcesink} does not necessarily hold; a one-point extension of a path algebra of Dynkin type is often of nonDynkin type.
This also means that the silting-discreteness of the both sides in a recollement is not always passed to the middle in the recollement.


We apply Example \ref{sourcesink} to tree quiver algebras.

\begin{example}\label{tree}
Let $\Lambda$ be a tree quiver algebra; i.e., the underlying graph of its Gabriel quiver is tree.
If $\Lambda$ is silting-discrete, then so is $\Lambda/\Lambda e\Lambda$ for every idempotent $e$ of $\Lambda$.
\end{example}

\subsection{Remark on Kronecker algebras}
As is well-known, the Kronecker algebras are a first and good example of $\tau$-tilting infinite algebras.
When we discuss silting-discreteness, one often needs graded Kronecker algebras, as we already observed in the preceding subsections.
Now, we summarize it.

Let $\KK$ be the dg quiver algebra of a graded $n$-Kronecker quiver with trivial differential.
For $\lambda:=(n_i)\in\bigoplus_{i\in\Z}\Z_{\geq0}$, we write $\KK:=\KK(\lambda)$, where $n_i$ stands for the number of arrows of degree $i$; we will omit the zero parts.
Note that $n=\sum_i n_i$.
We remark that applying shifts, $\KK((n_i)_i)$ is derived equivalent to $\KK((n_{i+m})_i)$ for every $m\in\Z$.

In case that $n_i\neq0$ for at most two components $i$, we get the following result.

\begin{proposition}\label{nKronecker}
Let $i<0$.
Then $\KK=\KK(n_0, n_i)$ is silting-discrete iff $n_0, n_i\leq1$.
\end{proposition}
\begin{proof}
If $n_0\geq2$, then $H^0(\KK)$ is the ordinary (2-)Kronecker algebra; so, it is not $\tau$-tilting finite.
Therefore, $\KK$ is not silting-discrete.
Let $n_i\geq2$ and $S:=H^0(\KK)/\rad H^0(\KK)$.
Then the Koszul dual $\KK^!:=\Ext_\KK^*(S, S)$ of $\KK$ is quasi-isomorphic to $\KK(n'_{1-i}, n'_1)$, which is derived equivalent to $\KK(n''_0, n''_i)$.
Here, $n''_0=n'_{1-i}=n_i$ and $n''_i=n'_1=n_0$.
Since $\KK$ and $\KK^!$ are derived equivalent \cite[Section 5]{SY}, we obtain that $\KK$ is not silting-discrete.

Let us show the `if' part; we only handle the case that $n_0=1=n_i$.
Let $\Lambda$ be the radical-square-zero algebra given by the quiver
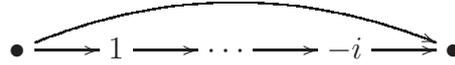
\[\xymatrix{
\bullet \ar@/^1.5pc/[rrrr]^{} \ar[r] & 1 \ar[r] & \cdots \ar[r] & -i \ar[r] & \bullet  
}\]
It follows from Proposition \ref{goc} that $\Lambda$ is silting-discrete.
Let $e_v$ be the primitive idempotent of $\Lambda$ corresponding to the vertex $v$.
Taking the idempotent $e:=e_1+\cdots +e_{-i}$, we see that $\KK$ is quasi-isomorphic to $\A_e$,
whence it is silting-discrete by Corollary \ref{heDG}.
\end{proof}

\begin{remark}
We know from \cite[Corollary 3.2]{LY2} that if $\KK$ is derived equivalent to an `ordinary' algebra, then $\KK=\KK(n_0, n_{-1})$ up to shift.
Moreover, by \cite[Corollary 3.10]{LY2} it is seen that $\KK(n_0, n_{-1})$ is derived equivalent to the algebra presented by the quiver consisting of vertices 1 and 2 with $n_0$ arrows from 1 to 2 (say $x$) and $n_{-1}$ arrows from 2 to 1 (say $y$), and with all relations $yx=0$.
This is just quasi-hereditary with 2 simples.
\end{remark}

\begin{remark}
We observe that $\KK(1,1)$ for an arbitrary degree is derived equivalent to a graded gentle one-cycle algebra, which behaves like (ordinary) derived-discrete algebras.
We hope that graded gentle one-cycle algebras which are not derived equivalent to $K\widetilde{A}$ are silting-discrete; see \cite{KY3}.
\end{remark}

\subsection{Remark on quotients}\label{sec:quot}

We give a remark on the inheritance of silting-discreteness by (idempotent) quotients, comparing with $\tau$-tilting finiteness.

First, note that algebra (`ordinary') epimorphisms can not necessarily transmit the silting-discreteness;
see Proposition \ref{AtildeC} and Example \ref{Atilde2}.

Next, let us observe that idempotent quotients can not always inherit the silting-discreteness, either.
Let $\Gamma$ be the algebra given by the quiver
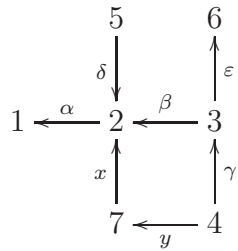
\[\xymatrix{
 & 5 \ar[d]_\delta & 6 \\
1 & 2 \ar[l]_\alpha & 3 \ar[l]_\beta\ar[u]_\varepsilon \\
 & 7 \ar[u]^x & 4 \ar[l]^y\ar[u]_\gamma
}\]
with relations $\beta\alpha=\delta\alpha=\gamma\varepsilon=0$ and $\gamma\beta=yx$.
To check that $\Gamma$ is a piecewise hereditary algebra of type $E_7$,
we first show the following claim.

\begin{claim*}
Let $\Lambda$ be a piecewise hereditary algebra of Dynkin type and $M$ an indecomposable right $\Lambda$-module.
Then the one-point extension algebra
$\begin{pmatrix}
K & M \\
0 & \Lambda
\end{pmatrix}$ is derived equivalent to a hereditary algebra.
\end{claim*}
\begin{proof}
Let $A$ be a hereditary algebra of Dynkin type which is derived equivalent to $\Lambda$.
Applying shift functors if needed, we may assume that $M$ corresponds to an indecomposable $A$-module $N$ under a derived equivalence between $\Lambda$ and $A$.
Since $A$ is representation-finite, there is an integer $\ell\geq0$ with $P:=\tau^\ell(N)$ projective.
Here, $\tau$ denotes the Auslander--Reiten translation.
As it induces a derived autoequivalence of $A$, we can take a derived equivalence between $\Lambda$ and $A$ which sends $M$ to $P$.
By \cite[Theorem 1]{BL}, we obtain that
$\begin{pmatrix}
K & M \\
0 & \Lambda
\end{pmatrix}$
is derived equivalent to 
$\begin{pmatrix}
K & P \\
0 & A
\end{pmatrix}$,
which is hereditary.
\end{proof}

Let us return our attension to $\Gamma$.
Since the full subquiver with the vertices 2, 3, 4 and 7 makes a piecewise hereditary algebra of type $D_4$, we observe by the claim that the one-point extension at the vertex 5 is derived equivalent to a hereditary algebra whose Coxeter polynomial is $(x+1)(x^4+1)$, which is of type $D_5$.
Therefore, it turns out that the algebra presented by the full subquiver with the vertices 2, 3, 4, 5 and 7 is piecewise hereditary of type $D_5$.
As a similar argument (use the dual of the claim), we see that $\Gamma$ is derived equivalent to a hereditary algebra with Coxeter polynomial $(x+1)(x^6-x^3+1)$, which is of type $E_7$. 
Thus, we figure out that $\Gamma$ is silting-discrete.

However, it is obtained by Example \ref{exampleY} that the idempotent quotient $\Gamma/\Gamma e_7\Gamma$ is not silting-discrete, where $e_7$ denotes the primitive idempotent corresponding to the vertex 7.


\begin{remark}
Although Coxeter polynomials are a derived invariant, we can not judge the silting-discreteness of a given algebra only by them, in general.
Actually, the extended canonical algebra of type $\langle 2,4,6 \rangle$, which is given by the quiver
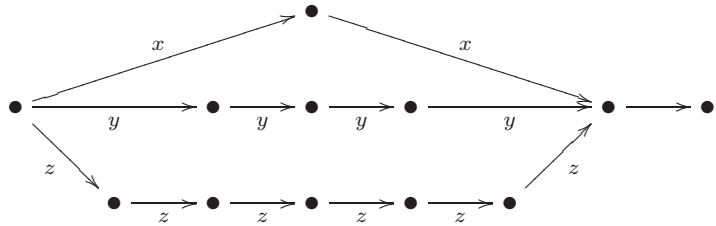
\[\xymatrix{
 & & & \bullet \ar[drrr]^x & & & & \\
\bullet \ar[urrr]^x \ar[rr]_y \ar[dr]_z & & \bullet \ar[r]_y  & \bullet \ar[r]_y & \bullet \ar[rr]_y & & \bullet \ar[r] & \bullet \\
 & \bullet \ar[r]_z & \bullet \ar[r]_z & \bullet \ar[r]_z & \bullet \ar[r]_z & \bullet \ar[ur]_z & & 
}\]
with relation $z^6=x^2-y^4$, has the same Coxeter polynomial as a (piecewise) hereditary algebra of type $D_{12}$ \cite[Proposition 5.9]{LP}, but it is not silting-discrete.
Note that the algebra is never derived equivalent to a hereditary algebra;
for a hereditary algebra, 
if the spectral radius ($=1$) is not a root of the Coxeter polynomial, then the algebra is representation-finite;
if the spectral radius ($=1$) is a root of the Coxeter polynomial, then the algebra is representation-tame;
if the spectral radius is strictly greater than 1, then the algebra is representation-wild.
\end{remark}

\section{Applications}\label{sec:applications}

In this section, we investigate the silting-discreteness of finite dimensional algebras; in particular, we focus on
\begin{itemize}
\item simply-connected tensor algebras (\ref{subsec:simplyconnected}), and 
\item selfinjective Nakayama algebras (\ref{subsec:selfinjectivenakayama}).
\end{itemize}

Let $\Lambda$ be a finite dimensional $K$-algebra which is basic and ring-indecomposable.
Let $\overrightarrow{A_n}$ be the quiver $\xymatrix{1 \ar[r] & 2 \ar[r] & \cdots \ar[r] & n}$.
For two successive arrows $\xymatrix{\bullet \ar[r]^\alpha & \bullet \ar[r]^\beta & \bullet}$ in the Gabriel quiver of $\Lambda$, we draw
$\xymatrix{
\bullet \ar[r]_{}="a" & \bullet \ar[r]_{}="b" & \bullet
\ar@/_1pc/@{.}"a";"b"_{}
}$
if $\alpha\beta=0$.

\subsection{Simply-connected tensor algebras}\label{subsec:simplyconnected}

In this subsection, we completely classify silting-discrete simply-connected tensor algebras.
Since we already know that the tensor algebra of three nonlocal algebras is anything but silting-discrete \cite[Proposition 4.1]{AH},
one turns the interest to the classification of silting-discrete simply-connected tensor algebras of two algebras.
Here is the first case; one of the two is $K$.

\begin{proposition}\label{simplyconnected}
A simply-connected algebra is silting-discrete if and only if it is piecewise hereditary of Dynkin type. 
\end{proposition}
\begin{proof}
The `if' part is trivial.
We show the `only if' part.
Let $\Lambda$ be a silting-discrete simply-connected algebra.
Note that the silting-discreteness implies the $2_\Lambda$-silting finiteness (the $\tau$-tilting finiteness for $\Lambda$).
We see by \cite[Theorem 3.4]{W} that $\Lambda$ is representation-finite.
Similarly, it follows from \cite[Lemma 3.1]{ANS} that every reflection of $\Lambda$ is simply-connected and $\tau$-tilting finite, so it is representation-finite.
Then, we apply Proposition 3.3 of \cite{ANS} to deduce that $\Lambda$ is piecewise hereditary of Dynkin type.
\end{proof}

\emph{
In the rest of this subsection, let $A$ and $B$ be triangular algebras 
which are basic and ring-indecomposable}, and put $\Lambda:=A\otimes_KB$.
Here, a \emph{triangular} algebra is a finite dimensional $K$-algebra whose Gabriel quiver is acyclic.

Our goal of this subsection is to give a complete classification of silting-discrete simply-connected tensor algebras.
We know that $\Lambda$ is silting-discrete if $A$ is local and $B$ is silting-discrete \cite[Theorem 2.1]{AH}.
Let us give an easy observation.

\begin{lemma}\label{ITforTensor}
If $\Lambda$ is silting-discrete, then so are both $A$ and $B$.
\end{lemma}
\begin{proof}
Let $e$ be a primitive idempotent of $A$.
As $A$ is triangular, we have an isomorphism $eAe\simeq K$.
Apply Corollary \ref{idempotent} to deduce that $B\simeq (e\otimes1)\Lambda(e\otimes1)$ is silting-discrete.
\end{proof}

We determine the algebra structure of one of the components when $\Lambda$ is silting-discrete.

\begin{lemma}\label{atmost2}
If $\Lambda$ is silting-discrete, then at least one of $A$ and $B$ has at most 2 nonisomorphic simple modules.
In particular, it is isomorphic to $K$ or $K\overrightarrow{A_2}$.
\end{lemma}
\begin{proof}
By Lemma \ref{ITforTensor}, it is observed that $A$ and $B$ are silting-discrete, and so they have no multiple arrow in their Gabriel quivers \cite[Theorem 5.12(d)]{DIRRT}.
Note that for every idempotent $e$ of $A$, $eAe$ is a silting-discrete triangular algebra.

Now, assume that $A$ has at least 3 nonisomorphic simple modules.
As above, there exists an idempoten $e$ of $A$ such that $eAe$ is isomorphic to one of the following:
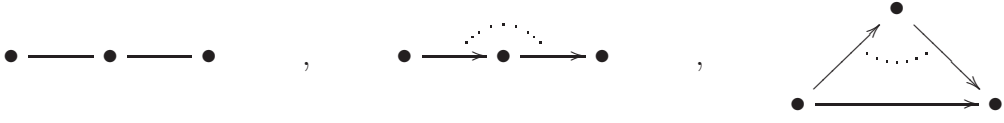
\[\begin{array}{c@{\hspace{1cm},\hspace{1cm}}c@{\hspace{1cm},\hspace{1cm}}c}
\vcenter{\xymatrix{
\bullet \ar@{-}[r] & \bullet \ar@{-}[r] & \bullet 
}} &
\vcenter{\xymatrix{
\bullet \ar[r]^{}="a" & \bullet \ar[r]^{}="b" & \bullet 
\ar@/^1pc/@{.}"a";"b"
}} & 
\vcenter{\xymatrix{
 & \bullet \ar[dr]_(0.4){}="b" &  \\
\bullet \ar[ru]_(0.6){}="a" \ar[rr] & & \bullet
\ar@/_0.5pc/@{.}"a";"b"
}}
\end{array}\]
From Corollary \ref{idempotent}, we obtain that $(e\otimes 1)\Lambda(e\otimes 1)=(eAe)\otimes_KB$ is silting-discrete.
Note that the first two of the three algebras above are derived equivalent to the path algebra $K\overrightarrow{A_3}$.
If $eAe$ is one of the first two, then $(eAe)\otimes_KB$ is derived equivalent to $K\overrightarrow{A_3}\otimes_KB$, which must be silting-discrete.
This yields by \cite[Theorem 4.4]{AH} that $B$ is a Nakayama algebra with radical square zero, whence it is isomorphic to $K\overrightarrow{A_r}/\rad^2K\overrightarrow{A_r}$ for some $r>0$; because $B$ is triangular.
We see that $(eAe)\otimes_KB$ is derived equivalent to $K\overrightarrow{A_3}\otimes_KK\overrightarrow{A_r}$,
which is also silting-discrete.
By \cite[Theorem 4.4]{AH} again, we have $r\leq2$.

Finally, let us suppose that $eAe$ is the last of the three algebras above, which is derived equivalent to the algebra $C$ given by the quiver
\[\xymatrix{
1 \ar[r]^\alpha & 2 \ar@<2pt>[r]^\beta & 3 \ar@<2pt>[l]^\gamma
}\]
with $\gamma\beta=0$.
So, we observe that $(eAe)\otimes_KB$ is derived equivalent to $C\otimes_KB$, which is silting-discrete.
Since there is an epimorphism $C\otimes_KB\to K\overrightarrow{A_3}\otimes_KB$,
it turns out that $K\overrightarrow{A_3}\otimes_KB$ is $\tau$-tilting finite \cite{DIRRT}.
As a similar argument above, we conclude that $B$ has at most 2 nonisomorphic simple modules.
\end{proof}

We often call the algebra $K\overrightarrow{A_2}\otimes_KK\overrightarrow{A_r}$ the \emph{commutative ladder} of degree $r$,
which is prenseted by the quiver with relations as follows:
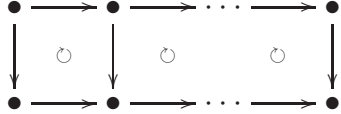
\[\xymatrix{
\bullet \ar[r]\ar[d]\ar@{}[dr]|\circlearrowright & \bullet \ar[r]\ar[d]\ar@{}[dr]|\circlearrowright & \cdots \ar[r]\ar@{}[dr]|\circlearrowright & \bullet \ar[d] \\
\bullet \ar[r] & \bullet \ar[r] & \cdots \ar[r] & \bullet 
}\]

Now, we realize the goal of this subsection, which gives a complete classification of silting-discrete simply-connected tensor algebras.

\begin{theorem}\label{sdscta}
Assume that $A$ and $B$ are nonlocal and simply-connected.
Then the following are equivalent:
\begin{enumerate}
\item $\Lambda$ is silting-discrete;
\item It is a piecewise hereditary algebra of type $D_4, E_6$ or $E_8$;
\item It is derived equivalent to the commutative ladder of degree $\leq 4$.
\end{enumerate}
\end{theorem}
\begin{proof}
The equivalence (2)$\Leftrightarrow$(3) is due to \cite{L}.
We know the implication (2)$\Rightarrow$(1) holds.
Let us show the implication (1)$\Rightarrow$(3) holds true.
Applying Lemma \ref{atmost2}, we may suppose that $A\simeq K\overrightarrow{A_2}$.
By Lemma \ref{ITforTensor}, one sees that $B$ is silting-discrete.
Therefore, it follows from Proposition \ref{simplyconnected} that $B$ is piecewise hereditary of Dynkin type $\Delta$,
whence $\Lambda$ is derived equivalent to $K\overrightarrow{A_2}\otimes_KK\overrightarrow{\Delta}$.
Since this is silting-discrete, we observe that $\Delta$ must be of type $A_r$ by \cite[Theorem 3.2]{AH}; that is, $\Lambda$ is derived equivalent to the commutative ladder of degree $r$.
By \cite[Example 3.3]{AH}, we get $r\leq4$. 
\end{proof}

\begin{remark}
Lemma \ref{atmost2} also says that if $\Lambda$ is silting-discrete, then at least one of $A$ and $B$ is automatically simply-connected.
Then, we ask whether the both are simply-connected or not.
In fact, it seems to be unknown that $K\overrightarrow{A_2}\otimes_KC$ is silting-discrete,
where $C$ is the algebra presented by the quiver with relation:
\[\xymatrix{
 & \bullet \ar[dr]_(0.4){}="b" &  \\
\bullet \ar[ru]_(0.6){}="a" \ar[rr] & & \bullet
\ar@/_0.5pc/@{.}"a";"b"
}\]
(It is $\tau$-tilting finite, thanks to Aoki's QPA programm.)
This is also one reason why we can not drop the assumption of $A$ and $B$ being simply-connected in Theorem \ref{sdscta}.
\end{remark}

\subsection{Selfinjective Nakayama algebras}\label{subsec:selfinjectivenakayama}

Ten years ago, the first-named author of this paper showed that any representation-finite symmetric algebra is silting-discrete \cite{Ai}.
Then, we might naturally hope that every representation-finite selfinjective algebra is also silting-discrte.
However, we here give the following surprising result, which says that the guess does not necessarily hold.


\begin{theorem}\label{Nakayama}
Let $\Lambda$ be a (nonlocal) selfinjective Nakayama algebra; that is, it is presented by the cyclic quiver with $n$ vertices, and the $r$-th radical is zero for some $n, r>1$.
Then $\Lambda$ is not silting-discrete if (i) $r=3, 4$ and $n\geq11$; (ii) $r=5,6$ and $n\geq r+8$; or (iii) $r\geq7$ and $n\geq 2r+1$.
\end{theorem}

Let $n,r>1$. We denote by $N_{n,r}$ the algebra given by the quiver 
\[\xymatrix{
1 \ar[r]^x & 2 \ar[r]^x & \cdots \ar[r]^x & n \ar@/^1pc/[lll]^x
}\]
with $x^r=0$.
The primitive idempotent corresponding to the vertex $i$ of the quiver is denoted by $e_i$.
As is well-known, an algebra is a nonlocal selfinjective (symmetric) Nakayama algebra if and only if it is isomorphic to $N_{n, r}$ for some $n,r>1$ ($r\equiv 1\pmod n$).

We give an example of silting-discrete selfinjective Nakayama algebras.

\begin{proposition}\label{sdsn}
If $r=2$ or $r\equiv1\pmod n$, then $N_{n,r}$ is silting-discrete. 
\end{proposition}
\begin{proof}
By Proposition \ref{goc} for $r=2$ and \cite[Theorem 5.6]{Ai} for $r\equiv1\pmod n$.
%
\end{proof}


Put $A(n,r):=K\overrightarrow{A_n}/\rad^r K\overrightarrow{A_n}$.
Under a suitable condition, we can get $A(n,r)$ by an idempotent truncation of a selfinjective Nakayama algebra. 

\begin{lemma}\label{LNfromCN}
Let $1\leq s\leq n$ and put $e:=e_1+\cdots+e_s$.
If $s+r\leq n+1$, then $eN_{n,r}e$ is isomorphic to $A(s,r)$.
\end{lemma}
\begin{proof}
Straightforward.
\end{proof}

Thanks to the list of \cite{HS} (see also \cite{LP}), we give a complete classification.

\begin{lemma}\label{A(n,r)}
\begin{enumerate}
\item If $A(n,r)$ is not silting-discrete, then neither is $A(n+1, r)$;
\item $A(n,r)$ is silting-discrete if and only if one of the following cases occurs:
(i) $r=2$;
(ii) $r=3, 5, 6$ and $n\leq 8$;
(iii) $r=4$ and $n\leq 7$;
(iv) $r\geq7$ and $n=r+1$.
\end{enumerate}
\end{lemma}

Now, we show the main theorem of this subsection.

\begin{proof}[Proof of Theorem \ref{Nakayama}]
Let $\Lambda:=N_{n,r}$ for $n,r>1$.
We put $s$ in each case as follows:
\begin{enumerate}[(i)]
\item $r=3,4$ and $n\geq 11\rightsquigarrow s=9\ (r=3)$ or $s=8\ (r=4)$;
\item $r=5,6$ and $n\geq r+8\rightsquigarrow s=9$;
\item $r\geq7$ and $n\geq 2r+1\rightsquigarrow s=r+2$.
\end{enumerate}
In all the cases, the assumption of Lemma \ref{LNfromCN} is satisfied; hence, $e\Lambda e\simeq A(s,r)$.

By Lemma \ref{A(n,r)}, we see that $A(s,r)$ is not silting-discrete for these pairs $(s,r)$, 
whence $\Lambda$ is not silting-discrete by Corollary \ref{idempotent}.
\end{proof}

\appendix
\section{Derived-discrete algebras are silting-discrete}
\label{subsec:ubipc}

The aim of this appendix is to give an approach to checking the silting-discreteness by classification of indecomposable objects.
As a corollary, we obtain that any derived-discrete algebra is silting-discrete \cite[Example 3.9, Corollary 4.2]{YY}; the case that it has finite global dimension is due to \cite{BPP}.
Here is the main theorem of this appendix.

\begin{theorem}\label{lft}
Assume that for each integer $d>0$, there exists an upper bound of the dimensions of $\Hom_\Lambda(\oplus_{i\in\Z}P^i, \Lambda/\rad \Lambda)$ for all indecomposable perfect complexes $P$ with length $d$.
Then $\Lambda$ is silting-discrete.
\end{theorem}
\begin{proof}
By \cite[Corollary 9]{HZS} (see also \cite[Theorem 1.1]{ANR}), the assumption implies  that there are only finitely many indecomposable presilting complexes of $\Kb(\proj\Lambda)$ with length $d$, which derives that $\dsilt{d}{\Lambda}\Kb(\proj\Lambda)$ is finite.
Thus, $\Kb(\proj\Lambda)$ is silting-discrete.
\end{proof}

For example, the assumption of Theorem \ref{lft} is satisfied if $\Lambda$ is a derived-disctete algebra \cite{BM, BGS} (see also \cite{ALPP}).
So, we recover the result of Yao--Yang.

\begin{corollary}\label{deriveddiscrete}
Any derived-discrete algebra is silting-discrete.
In particular, every Nakayama algebra with radical square zero is silting-discrete.
\end{corollary}
\begin{proof}
Let $\Lambda$ be a derived-discrete algebra, but not piecewise hereditary of Dynkin type. 
By \cite{BGS}, we see that $\Lambda$ is derived equivalent to the algebra presented by the quiver with relations for some $1\leq r\leq n$ and $m\geq 0$:
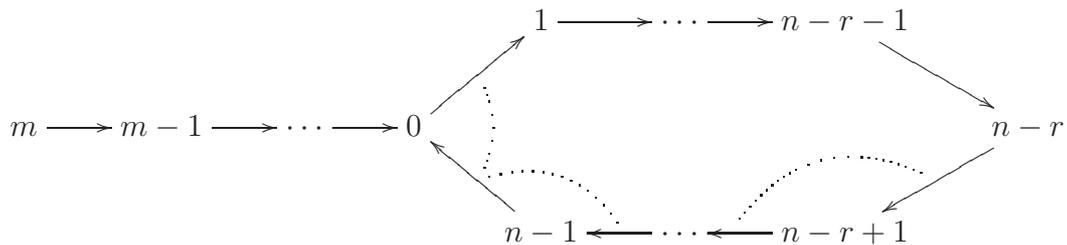
\[\xymatrix{
   &        &           &     & 1 \ar[r] &  \cdots \ar[r]  & n-r-1 \ar[rd] & \\
m \ar[r] & m-1 \ar[r] & \cdots \ar[r] & 0 \ar[ru]_{}="b"  &    &             &         & n-r \ar[ld]_{}="e"  \\
   &       &            &     & n-1 \ar[ul]_{}="a" &  \cdots \ar[l]_(0.4){}="c" & n-r+1 \ar[l]_(0.7){}="d" & 
\ar@/_/@{.}"a";"b"
\ar@/_1pc/@{.}"c";"a"
\ar@/_1.5pc/@{.}"e";"d"
}\]
It follows from \cite{BM} that each term of an indecomposable perfect complex is a multiplicity-free projective module,
which means that $\Lambda$ satisfies the assumption of Theorem \ref{lft}, so it is silting-discrete.
\end{proof}


\section{Triangulated categories whose silting objects are tilting}

The gap between siltings and tiltings has been often drawing attention lately; see \cite{AdK, AD} for example. 
In this appendix, we explore when a triangulated category has only tilting objects;
the main result of this appendix was used by the first-named author of this paper in the article \cite{Ai2}, but he gave no proof because it is easy. 
The aim of this appendix is to present a proof and examples of the main result.

We say that $\T$ is \emph{asotic} if any silting object in $\T$ is tilting; The word `asotic' is an abbreviation of the phrase `Any Silting Object is Tilting' $+$ the suffix `-ic'.

In this appendix, algebras are always finite dimensional $K$-algebras which are basic and ring-indecomposable, unless otherwise noted.

The first example of asotic triangulated categories is a triangulated category with an indecomposable tilting object \cite[Theorem 2.26]{AI}.
In the following, let us consider the nonlocal case.

We say that $\T$ is \emph{$\ell$-Calabi--Yau} if there is a bifunctorial isomorphism $\Hom_\T(-,?)\simeq D\Hom_\T(?,-[\ell])$,
where $D$ stands for the $K$-dual.
Here is an easy example of asotic triangulated categories;
see \cite[Lemma 2.7 and Example 2.8]{AI}.

\begin{example}\label{0CY}
A 0-Calabi--Yau triangulated category is asotic.
In particular, the perfect derived cagegory of a symmetric algebra is 0-Calabi--Yau, and so it is asotic.
\end{example}

We also know that a (complete) preprojective algebra of extended Dynkin type admits an asotic perfect derived category \cite[Proposition A.1]{KM}; note that the algebra is not finite dimensional and the perfect derived category has symmetry like 0-Calabi-Yau property. 

The main result of this appendix gives a slight generalization of Example \ref{0CY}.
Leaving the details for silting mutation to the paper \cite{AI}, we use the terminologies $\mu_X^-(T)$ and $\mu_X^+(T)$ for the left and right mutations of a silting object $T$ at its direct summand $X$.

\begin{theorem}\label{topsocle}
Let $\Lambda$ be an algebra and $\T:=\Kb(\proj\Lambda)$.
\begin{enumerate}
\item The following are equivalent:
\begin{enumerate}
\item $\mu_P^-(\Lambda)$ is tilting for any indecomposable projective module $P$;
\item The top of the indecomposable injective module corresponding to $S$ is the direct sum of some copies of $S$ for every simple module $S$. 
\end{enumerate}
\item The following are equivalent:
\begin{enumerate}
\item $\mu_P^+(\Lambda)$ is tilting for any indecomposable projective module $P$;
\item The socle of the indecomposable projective module corresponding to $S$ is the direct sum of some copies of $S$ for every simple module $S$. 
\end{enumerate}
\end{enumerate}
\end{theorem}
\begin{proof}
We show (2); the other can be handled dually.
Let $S$ be a simple module and $P$ its corresponding indecomposable projective module which is of the form $e\Lambda$ for a primitive idempotent $e$ of $\Lambda$.
Then, we observe that the right mutation $\mu_P^+(\Lambda)$ is the direct sum of a complex $X$ and the stalk complex $\Lambda/P$ concerned in degree 0,
where $X$ is the $(-1)$-shift of the projective presentation of $M:=e\Lambda/e\Lambda(1-e)\Lambda$.
Note that the second syzygy $\Omega^2(M)$ of $M$ is contained in the direct sum of some copies of $\Lambda/P$ \cite[Lemma 2.25]{AI}.
Thus, the equivalence (i)$\Leftrightarrow$(ii) can be checked easily by an isomorphism
$\Hom_{\Kb(\proj\Lambda)}(\mu_P^+(\Lambda), \mu_P^+(\Lambda)[-1])\simeq \Hom_\Lambda(M, \Lambda/P\oplus \Omega^2(M))$.
\end{proof}

For example, weakly-symmetric algebras satisfy the condition (ii) as in Theorem \ref{topsocle}(1)(2).
Moreover, the weakly-symmetric property of algebras is a derived invariant; see \cite[Proposition 3.1]{AD} for example.
These yield the following corollary.

\begin{corollary}\label{asoticws}
Let $\Lambda$ be a weakly-symmetric algebra and $\T:=\Kb(\proj\Lambda)$.
Let $\C$ be a connected component of the Hasse quiver of $\silt\T$.
Then the following hold:
\begin{enumerate}
\item If $\C$ has a tilting object, then all members in $\C$ are tilting.
\item If $\Lambda$ is silting-connected; that is, $\C=\silt\T$, then $\T$ is asotic.
\end{enumerate}
\end{corollary}

We give an example of nonsymmetric algebras whose perfect derived categories are asotic; see \cite{Ai2, AM, AdK, AD}.

\begin{example}
Let $\Lambda$ be the preprojective algebra of Dynkin type $D_{2n}, E_7$ or $E_8$, which is nonsymmetric weakly-symmetric.
Then $\T:=\Kb(\proj\Lambda)$ is asotic.
\end{example}
\begin{proof}
The strategy is due to \cite{AM}; we give a proof here for the convenience of the reader.

By Theorem \ref{topsocle}, any (irreducible) silting mutation of $\Lambda$ is a tilting object whose endomorphism algebra is isomorphic to $\Lambda$.
So, for every sequence $\Lambda=:T_0, T_1, \cdots,T_d$ of silting objects such that $T_{i+1}$ is the (left) silting mutation of $T_i$ at an indecomposable direct summand, we see that all $T_i$'s are tilting.
Since the set $\dsilt{2}{T_i}\T$ is finite, we obtain that $\T$ is silting-discrete; in particular, it is silting-connected.
Thus, it turns out that $\T$ is asotic by Corollary \ref{asoticws}.
\end{proof}

We remark that there are weakly-symmetric algebras whose perfect derived categories are not asotic \cite[Section 4]{AD}.
This says that for such an algebra $\Lambda$, the Hasse quiver of $\silt(\Kb(\proj\Lambda))$ has a connected component consisting of nontilting silting objects.
On the other hand, we know from \cite[Proposition 3.6]{AD} that
for a weakly-symmetric algebra $\Lambda$, every silting object lying in $\dsilt{2}{\Lambda}(\Kb(\proj\Lambda))$ is tilting.

We also find nonweakly-symmetric algebras with asotic perfect derived categories.

\begin{proposition}
Let $R$ be a local algebra and $\Lambda$ a silting-discrete algebra.
If $\Kb(\proj\Lambda)$ is asotic, then so is $\Kb(\proj\Lambda\otimes_KR)$.
\end{proposition}
\begin{proof}
The assertion follows from \cite[Theorem 2.1]{AH}.
\end{proof}

\subsection{Thick subcategories}

As an application, we describe thick subcategories generated by silting objects (i.e., the thick closures of presilting objects) in terms of algebras associated to a given algebra; however, we will assume the Bongartz-type condition (i.e., every presilting object is partial silting).
Since any silting-discrete triangulated category satisfies the Bongartz-type condition \cite[Theorem 2.15]{AM},
we can, for example, choose a silting-discrete symmetric algebra as our algebra.

The following proposition is practical to write out full triangulated subcategories with silting objects (up to triangle equivalence).

\begin{proposition}\label{fts}
Let $\Lambda$ be an algebra whose perfect derived category is asotic and satisfies the Bongartz-type condition.
Then every full triangulated subcategory of $\Kb(\proj\Lambda)$ with a silting object can be realized as $\Kb(\proj e\Gamma e)$, up to triangle equivalence.
Here, $\Gamma$ is a derived equivalent algebra to $\Lambda$ and $e$ is its idempotent.
\end{proposition}
\begin{proof}
Let $\U$ be a full triangulated subcategory of $\T:=\Kb(\proj\Lambda)$ with $U$ silting.
Since $U$ is presilting in $\Kb(\proj\Lambda)$, the Bongartz-type condition of $\T$ implies that $U$ can be completed to a silting object $T:=U\oplus X$ of $\Kb(\proj\Lambda)$.
As $\T$ is asotic, it is seen that $T$ is tilting in $\T$, so $U$ is tilting in $\U$.
Hence, $\Gamma:=\End_\T(T)$ is derived equivalent to $\Lambda$.
Take the composition $e$ of the canonical morphisms $T\to U\to T$, which is an idempotent of $\Gamma$ with $e\Gamma e\simeq \End_\T(U)$.
Since $U$ is tilting in $\U$, it turns out that there are triangle equivalences $\U\simeq\Kb(\proj\End_\T(U))\simeq \Kb(\proj e\Gamma e)$.
\end{proof}

Let us give an easy example.

\begin{example}
Let $\Lambda:=\Lambda_0$ be the multiplicity-free Brauer star algebra with 3 edges; i.e., its Brauer tree is: 
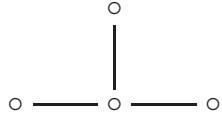
\[\xymatrix{
 & \circ & \\
\circ \ar@{-}[r] & \circ \ar@{-}[r]\ar@{-}[u] & \circ
}\]
As is well-known \cite{R2}, there are two derived equivalent algebras to $\Lambda$; one is $\Lambda$ itself and the other is the  multiplicity-free Brauer line algebra $\Lambda_1$ with 3 edges.
Taking idempotent truncations, we get 3 kinds of Brauer tree algebras other than $\Lambda$ and $\Lambda_1$; the Brauer tree algebras $\Lambda_2, \Lambda_3$ and $\Lambda_4$ whose Brauer trees are $G_2$, $G_3$ and $G_3\times G_3$, respectively:
\[\begin{array}{c@{\hspace{1cm},\hspace{1cm}}c}
G_2:=\vcenter{\xymatrix{
\circ \ar@{-}[r] & \circ \ar@{-}[r] & \circ
}} &
G_3:=\vcenter{\xymatrix{
\circ \ar@{-}[r] & \circ 
}}
\end{array}\]
By Proposition \ref{fts}, these give all (nonzero) full triangulated subcategories of $\Kb(\proj\Lambda)$ with siltings, which are triangle equivalent to $\Kb(\proj\Lambda_i)$ for $i\in\{0,1,2,3,4\}$.
\end{example}

Even if a given full triangulated subcategory has a tilting object, it does not necessarily poccess partial tilting in the whole; the following example was first appeared in \cite{RS}.

\begin{example}\label{Bongartzfails}
Let $\Lambda$ be the algebra presented by the quiver 
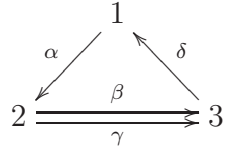
\[\xymatrix{
 & 1 \ar[dl]_\alpha & \\
2 \ar@<2pt>[rr]^\beta\ar@<-2pt>[rr]_\gamma & & 3 \ar[ul]_\delta
}\]
with relations $\alpha\beta=\gamma\delta=\delta\alpha=0$.
Then, $\Lambda$ has global dimension 4, and the simple module $S$ corresponding to the vertex 1 is a partial tilting module with projective dimension 2; hence it is pretilting in $\Kb(\proj\Lambda)$.
Let $\U$ be the thick closure of $S$ in $\Kb(\proj\Lambda)$.
Then, $\U$ has a tilting object $S$, but its (pre)tilting objects are never partial tilting in $\Kb(\proj\Lambda)$.  

Indeed, the Grothendieck group of $\U$ has rank 1, so $\silt\U=S[\Z]$.
However, we obtain from \cite[Example 4.4]{LVY} that for any $n\in\Z$, $S[n]$ is not partial tilting in $\Kb(\proj\Lambda)$.
Thus, $\U$ has no partial tilting object of $\Kb(\proj\Lambda)$.
On the other hand, $S$ is partial silting in $\Kb(\proj\Lambda)$; actually, $S\oplus \Lambda/P[2]$ is silting in $\Kb(\proj\Lambda)$,
where $P$ stands for the indecomposable projective module corresponding to the vertex 1.
%
\end{example}


\begin{remark}
A crux of Proposition \ref{fts} is that our triangulated category satisfies the Bongartz-type condition but not asotic; that is, the classification of thick subcategories with siltings can be done by using idempotent truncations of dg algebras.
However, we here avoided doing that because it is very difficult to classify all derived equivalent dg algebras to the original.
In this case, the asoticness condition was useful.
\end{remark}

\subsection{Remark on Bongartz-type condition}\label{Bongartzfails2}

Recently, it was pointed out in \cite{LZ} that the Bongartz-type condition for \emph{presiltings} does not necessarily hold.
Let us recall it here.

We consider the algebra $\Gamma$ presented by the quiver $\xymatrix{2 \ar@<2pt>[r]^x \ar@<-2pt>[r]_y & 1 \ar@<2pt>[r]^x \ar@<-2pt>[r]_y & 3}$ with $x^2=0=y^2$.
Let $S:=P_2/(y)$; it has projective dimension 2 and is pretilting in $\Kb(\proj\Gamma)$.
Here, $P_i$ stands for the indecomposable projective module corresponding to the vertex $i$.
Now, we find a thick generator $S\oplus P_2\oplus P_3$ of $\Kb(\proj\Gamma)$ whose dg endomorphism algebra $\Lambda$ is given by the quiver as in Example \ref{Bongartzfails} with $\deg(\delta)=2$;
by Keller--Rickard theorem, there is a derived equivalence between $\Gamma$ and $\Lambda$ which sends $S_\Gamma$ to ${P_1}_\Lambda$.
Applying the silting reduction (Theorem \ref{recollement} and Corollary \ref{Bongartz}) to $\Lambda$, we get isomorphisms $\silt_S\Gamma\simeq \silt_{P_1}\Lambda\simeq\silt\A$ of posets.
Here, $\A$ is the dg quiver algebra given by
$\xymatrix{
2 \ar@/^0.8pc/[r]^\beta\ar@/_0.8pc/[r]_\gamma & 3 \ar[l]|\varepsilon
}$
with $\gamma\varepsilon=0=\varepsilon\beta$ and $\deg(\varepsilon)=1$ (trivial differential).
It was proved that $\silt_S\Gamma=\emptyset$ in \cite{LZ} and $\silt\A=\emptyset$ in \cite{CJS}; that is, $S$ is not partial silting.


\begin{remark}
As seen above, silting reduction sometimes makes a triangulated category without silting.
We give a list of triangulated categories without silting:
\begin{itemize}
\item the singularity category $\ds(\Lambda)$ for a finite dimensional algebra $\Lambda$; in particular, the stable module category $\smod\Lambda$ for a selfinjective algebra $\Lambda$ \cite{CLZZ, AHMW, AI};
\item a positively-Calabi--Yau triangulated category \cite{AI};
\item the perfect derived category $\per(\A)$ for $\A$ as above with the same relations but with $\deg(\beta)+\deg(\varepsilon)=1=\deg(\gamma)+\deg(\varepsilon)$  \cite{CJS, JSW}.
\end{itemize}
\end{remark}

\section*{Acknowledgements}
The authors would like to give their gratitude to Osamu Iyama and Norihiro Hanihara for useful discussions and giving a lot of valuable and helpful comments.


\begin{thebibliography}{AAAAA}


\bibitem[AAC]{AAC}
{\sc T. Adachi, T. Aihara and A. Chan},
Classification of two-term tilting complexes over Brauer graph algebras.
{\it Math. Z.} {\bf 290} (2018), no. 1--2, 1--36.

\bibitem[AK]{AdK}
{\sc T. Adachi and R. Kase},
Examples of tilting-discrete self-injective algebras which are not silting-discrete.
Preprint (2020), arXiv: 2012.14119.

\bibitem[AIR]{AIR}
{\sc T. Adachi, O. Iyama and I. Reiten},
$\tau$-tilting theory.
{\it Compos. Math.} {\bf 150}, no. 3, 415--452 (2014).

\bibitem[AMY]{AMY}
{\sc T. Adachi, Y. Mizuno and D. Yang},
Discreteness of silting objects and $t$-structures in triangulated categories.
{\it Proc. Lond. Math. Soc. (3)} {\bf 118} (2019), no. 1, 1--42.

\bibitem[Ai1]{Ai}
{\sc T. Aihara},
Tilting-connected symmetric algebras.
{\it Algebr. Represent. Theory} {\bf 16} (2013), no. 3, 873--894.

\bibitem[Ai2]{Ai2}
{\sc T. Aihara},
On silting-discrete triangulated categories.
{\it Proceedings of the 47th Symposium on Ring Theory and Representation Theory}, 7--13, {\it Symp. Ring Theory Represent. Theory Organ. Comm., Okayama}, 2015.



\bibitem[AH]{AH}
{\sc T. Aihara and T. Honma},
$\tau$-tilting finite triangular matrix algebras.
{\it J. Pure Appl. Algebra} {\bf 225} (2021), no. 12, Paper No. 106785, 10pp.

\bibitem[AHMW]{AHMW}
{\sc T. Aihara, T. Honma, K. Miyamoto and Q. Wang},
Report on the finiteness of silting objects.
{\it Proc. Edinb. Math. Soc. (2)} {\bf 64} (2021), no. 2, 217--233.

\bibitem[AI]{AI}
{\sc T. Aihara and O. Iyama},
Silting mutation in triangulated categories.
{\it J. Lond. Math. Soc. (2)} {\bf 85} (2012), no. 3, 633--668.

\bibitem[AM]{AM}
{\sc T. Aihara and Y. Mizuno},
Classifying tilting complexes over preprojective algebras of Dynkin type.
{\it Algebra Number Theory} {\bf 11} (2017), no. 6, 1287--1315.

\bibitem[ANR]{ANR}
{\sc S. Al-Nofayee and J. Rickard},
Rigidity of tilting complexes and derived equivalence for self-injective algebras.
\url{http://www.maths.bris.ac.uk/~majcr/papers.html} (preprint)

\bibitem[AHMV]{AHMV}
{\sc L. Angeleri-Hugel, F. Marks and J. Vitoria},
Partial silting objects and smashing subcategories.
{\it Math. Z.} {\bf 296} (2020), no. 3--4, 887--900.

\bibitem[ALPP]{ALPP}
{\sc K. K. Arnesen, R. Laking, D. Pauksztello and M. Prest},
The Ziegler spectrum for derived-discrete algebras.
{\it Adv. Math.} {\bf 319} (2017), 653--698.

\bibitem[ANS]{ANS}
{\sc I. Assem, J. Nehring and A. Skowronski},
Domestic trivial extensions of simply connected algebras.
{\it Tsukuba J. Math.} {\bf 13} (1989), no. 1, 31--72.


\bibitem[AS]{AS}
{\sc I. Assem and A. Skowronski},
Iterated tilted algebras of type $\widetilde{A_n}$.
{\it Math. Z.} {\bf 195} (1987), no. 2, 269--290.



\bibitem[Au]{Au}
{\sc J. August},
On the finiteness of the derived equivalence classes of some stable endomorphism rings.
{\it Math. Z.} {\bf 296} (2020), no. 3--4, 1157--1183.

\bibitem[AD]{AD}
{\sc J. August and A. Dugas},
Silting and tilting for weakly symmetric algebras.
Preprint (2021), arXiv: 2101.03097.




\bibitem[BL]{BL}
{\sc M. Barot and H. Lenzing},
One-point extensions and derived equivalence.
{\it J. Algebra} {\bf 264} (2003), no. 1, 1--5.

\bibitem[BM]{BM}
{\sc V. Bekkert and H. A. Merklen},
Indecomposables in derived categories of gentle algebras.
{\it Algebr. Represent. Theory} {\bf 6} (2003), no. 3, 285--302.




\bibitem[BGS]{BGS}
{\sc G. Bobinski, C. Geiss and A. Skowronski},
Classification of discrete derived categories.
{\it Cent. Eur. J. Math.} {\bf 2} (2004), no. 1, 19--49.

\bibitem[BPP]{BPP}
{\sc N. Broomhead, D. Pauksztello and D. Ploog},
Discrete derived categories II: the silting pairs CW complex and the stability manifold.
{\it J. Lond. Math. Soc. (2)} {\bf 93} (2016), no. 2, 273--300.



\bibitem[CJS]{CJS}
{\sc W. Chang, H. Jin and S. Schroll},
Recollements of derived categories of graded gentle algebras and surface cuts.
Preprint (2022), arXiv: 2206.11196.


\bibitem[CLZZ]{CLZZ}
{\sc X.-W. Chen, Z.-W. Li, X. Zhang and Z. Zhao},
A non-vanishing result on the singularity category.
Preprint (2023), arXiv: 2301.01897.


\bibitem[DIJ]{DIJ}
{\sc L. Demonet, O. Iyama and G. Jasso},
$\tau$-tilting finite algebras, bricks and $g$-vectors.
{\it Int. Math. Res. Not. IMRN} 2019, no. 3, 852--892.

\bibitem[DIRRT]{DIRRT}
{\sc L. Demonet, O. Iyama, N. Reading, I. Reiten and H. Thomas},
Lattice theory of torsion classes.
Preprint (2017), arXiv: 1711.01785.

\bibitem[EJR]{EJR}
{\sc F. Eisele, G. Janssens and T. Raedschelders},
A reduction theorem for $\tau$-rigid modules.
{\it Math. Z.} {\bf 290} (2018), no. 3--4, 1377--1413.





\bibitem[HS]{HS}
{\sc D. Happel and U. Seidel},
Piecewise hereditary Nakayama algebras.
{\it Algebr. Represent. Theory} {\bf 13} (2010), no. 6, 693--704.



\bibitem[HZS]{HZS}
{\sc B. Huisgen-Zimmermann and M. Saorin},
Geometry of chain complexes and outer automorphisms under derived equivalence.
{\it Trans. Amer. Math. Soc.} {\bf 353} (2001), no. 12, 4757--4777.

\bibitem[IY]{IY}
{\sc O. Iyama and D. Yang},
Silting reduction and Calabi-Yau reduction of triangulated categories.
{\it Trans. Amer. Math. Soc.} {\bf 370} (2018), no. 11, 7861--7898.


\bibitem[JSW]{JSW}
{\sc H. Jin, S. Schroll and Z. Wang},
A complete derived invariant and silting theory for graded gentle algebras.
Preprint (2023), arXiv: 2303.17474.

\bibitem[KY1]{KY}
{\sc M. Kalck and D. Yang},
Relative singularity categories I: Auslander resolutions.
{\it Adv. Math.} {\bf 301} (2016), 973--1021.

\bibitem[KY2]{KY3}
{\sc M. Kalck and D. Yang},
Derived categories of graded gentle one-cycle algebras.
{\it J. Pure Appl. Algebra} {\bf 222} (2018), no. 10, 3005--3035.

\bibitem[KY3]{KY2}
{\sc M. Kalck and D. Yang},
Relative singularity categories II: dg modules.
Preprint (2018), arXiv: 1803.08192.

\bibitem[K]{K}
{\sc B. Keller},
Deriving DG categories.
{\it Ann. Sci. Ecole Norm. Sup.} (4) {\bf 27} (1994), 63--102.

\bibitem[KM]{KM}
{\sc Y. Kimura and Y. Mizuno},
Two-term tilting complexes for preprojective algebras of non-Dynkin type.
{\it Comm. Algebra} {\bf 50} (2022), no. 2, 556--570.


\bibitem[L]{L}
{\sc S. Ladkani},
On derived equivalences of lines, rectangles and triangles.
{\it J. Lond. Math. Soc. (2)} {\bf 87} (2013), no. 1, 157--176.

\bibitem[LP]{LP}
{\sc H. Lenzing and J. A. de la Pena},
Spectral analysis of finite dimensional algebras and singularities.
{\it Trends in representation theory of algebras and related topics}, 541--588, EMS Ser. Congr. Rep., {\it Eur. Math. Soc., Zurich}, 2008.



\bibitem[LVY]{LVY}
{\sc Q. Liu, J. Vitoria and D. Yang},
Gluing silting objects.
{\it Nagoya Math. J.} {\bf 216} (2014), 117--151.


\bibitem[LY]{LY2}
{\sc Q. Liu and D. Yang},
Stratifications of algebras with two simple modules.
{\it Forum Math.} {\bf 28} (2016), no. 1, 175--188.

\bibitem[LZ]{LZ}
{\sc Y.-Z. Liu and Y. Zhou},
A negative answer to complement questions for presilting complexes.
Preprint (2023), arXiv: 2302.12502.






\bibitem[N]{N}
{\sc A. Neeman},
The connection between the $K$-theory localization theorem of Thomason, Trobauch and Yao and the smashing subcategories of Bousfield and Ravenel.
{\it Ann. Sci. Ecole Norm. Sup. (4)} {\bf 25} (1992), no. 5, 547--566.

\bibitem[NS]{NS}
{\sc P. Nicolas and M. Saorin},
Parametrizing recollement data for triangulated categories.
{\it J. Algebra} {\bf 322} (2009), no. 4, 1220--1250.

\bibitem[O]{O}
{\sc S. Oppermann},
Quivers for silting mutation.
{\it Adv. Math.} {\bf 307} (2017), 684--714.

\bibitem[P]{P}
{\sc D. Pauksztello},
Homological epimorphisms of differential graded algebras.
{\it Comm. Algebra} {\bf 37} (2009), no. 7, 2337--2350.

\bibitem[PSZ]{PSZ}
{\sc D. Pauksztello, M. Saorin and A. Zvonareva},
Contractibility of the stability manifold of silting-discrete algebras.
{\it Forum Math.} {\bf 30} (2018), no. 5, 1255--1263.


\bibitem[R1]{R}
{\sc J. Rickard},
Morita theory for derived categories.
{\it J. London Math. Soc. (2)} {\bf 39} (1989), no. 3, 436--456.

\bibitem[R2]{R2}
{\sc J. Rickard},
Derived categories and stable equivalence.
{\it J. Pure Appl. Algebra} {\bf 61} (1989), no. 3, 303--317.

\bibitem[RS]{RS}
{\sc J. Rickard and A. Schofield},
Cocovers and tilting modules.
{\it Math. Proc. Cambridge Philos. Soc.} {\bf 106} (1989), no. 1, 1--5. 


\bibitem[S]{S}
{\sc L. Silver},
Noncommutative localizations and applications.
{\it J. Algebra} {\bf 7} (1967), 44--76.




\bibitem[SY]{SY}
{\sc H. Su and D. Yang},
From simple-minded collections to silting objects via Koszul duality.
{\it Algebr. Represent. Theory} {\bf 22} (2019), no. 1, 219--238.

\bibitem[V]{V}
{\sc D. Vossieck},
The algebras with discrete derived category.
{\it J. Algebra} {\bf 243} (2001), no. 1, 168--176.

\bibitem[W]{W} 
{\sc Q. Wang},
On $\tau$-tilting finite simply connected algebras.
{\it Tsukuba J. Math.} {\bf 46} (2022), no. 1, 1--37.


\bibitem[Y]{Y}
{\sc D. Yang},
Some examples of $t$-structures for finite-dimensional algebras.
{\it J. Algebra} {\bf 560} (2020), 17--47.

\bibitem[YY]{YY}
{\sc L. Yao and D. Yang},
The equivalence of two notions of discreteness of triangulated categories.
{\it Algebr. Represent. Theory} {\bf 24} (2021), no. 5, 1295--1312.



\end{thebibliography}
\end{document}